\documentclass[11pt,english]{amsart}
\usepackage[T1]{fontenc}
\usepackage[latin9]{inputenc}
\usepackage{amstext}
\usepackage{amsthm}
\usepackage{esint}

\makeatletter
\usepackage{amsmath,amsfonts,amssymb,amsthm,epsfig}
\usepackage{graphicx}
\usepackage{soul}

\usepackage{color}
\usepackage[final,allcolors=blue,colorlinks=true]{hyperref}

\def\R{\mathbb R}

\def\S{\mathbb S}

\newcommand{\HH}{{\mathcal H}}
\newcommand{\sing}{\text{\rm sing}}

\def\be{\beta}

\def\de{\delta}
\def\ep{\epsilon}
\def\la{\lambda}

\def\var{\varphi}
\def\om{\omega}
\def\na{\nabla}
\def\Om{\Omega}  
\def\De{\Delta}      
\def\Si{\Sigma}      
\def\La{\Lambda}      

\def\cal{\mathcal}
\def\wq{\infty}
\def\pa{\partial}

\def\loc{\text{\rm loc}}

\newcommand{\D}{{\rm d}}

\newcommand{\divv}{\text{\rm div}}
\newcommand{\medint}{-\kern -,375cm\int}         
\newcommand{\medintinrigo}{-\kern -,315cm\int}
\newcommand{\wto}{\rightharpoonup}                

\newcommand{\LLcorner}{\scalebox{1.5}{\ensuremath{\llcorner}}}   

\numberwithin{equation}{section}
\newtheorem{theorem}{Theorem}[section]
\newtheorem*{theorem*}{Theorem}  

\newtheorem*{conclusion*}{Conclusin}

\newtheorem*{corollary*}{Corollary}

\newtheorem{definition}[theorem]{Definition}

\newtheorem{lemma}[theorem]{Lemma}
\newtheorem*{lemma*}{Lemma}

\newtheorem*{notation*}{Notation}

\newtheorem{proposition}[theorem]{Proposition}
\newtheorem*{proposition*}{Proposition}

\newtheorem*{remark*}{Remark}

\newtheorem*{example*}{Example}                
\theoremstyle{definition}

\makeatother

\usepackage{babel}
\begin{document}
	\title[]{Energy identity for stationary biharmonic mappings into spheres in supercritical dimensions}
	
	\author[C.-Y. Guo, C. Wang and C.-L. Xiang]{Chang-Yu Guo, Changyou Wang and Chang-Lin Xiang$^*$}
	
	\address[Chang-Yu Guo]{Research Center for Mathematics and Interdisciplinary Sciences, Shandong University 266237,  Qingdao, P. R. China and and  Frontiers Science Center for Nonlinear Expectations, Ministry of Education, P. R. China and Department of Physics and Mathematics, University of Eastern Finland, 80101, Joensuu, Finland} \email{changyu.guo@sdu.edu.cn}
	\address[Changyou Wang]{Department of Mathematics, Purdue University, West Lafayette, IN 47907, USA} \email{wang2482@purdue.edu}
	\address[Chang-Lin Xiang]{Three Gorges Mathematical Research Center, China Three Gorges University,  443002, Yichang,  P. R. China}
	\email{changlin.xiang@ctgu.edu.cn}
	\thanks{*: Corresponding author}
	\thanks{ Guo is supported by the Young Scientist Program of the Ministry of Science and Technology of China (No.~2021YFA1002200), the Taishan Scholar Project and the NSF of Shandong Province (No.~ZR2022YQ01). Wang is partially supported by NSF DMS 2453789 and Simons Travel Grant TSM-00007723. The corresponding author	Xiang is financially supported by NSFC (No. 12271296),   the NSF of Hubei province (No. 2024AFA061). Guo and Xiang are also supported by the Jiangsu Provincial Scientific Research Center of Applied Mathematics under Grant No. 7707014040A.}
	
	\begin{abstract}
		Energy identity for harmonic type maps in supercritical dimensions is an important and difficult problem. For sphere-valued harmonic maps, the first breakthrough was achieved by Lin-Rivi\`ere  [Duke Math. J. 2002].  In this paper, by adapting their strategy, we establish the energy identity for stationary biharmonic maps into spheres in supercritical dimensions $n\ge 5$.
	\end{abstract}
	
	\maketitle
	
	{\small
		\smallskip
		\newline
	}
	\bigskip
	
	\section{Introduction and main results}
	
Energy identity for harmonic type maps in supercritical dimensions is an important and difficult problem. The first breakthrough was achieved in 2002 by Lin-Rivi\`ere \cite{Lin-Riviere-2002} for sphere-valued stationary harmonic maps. Since then, it has been expected  that their result should hold for other harmonic type maps. In this work, we aim to extend their result to  biharmonic maps, which can be viewed as  a natural higher order extension of harmonic maps, from an Euclidean  domain $\Om\subset\R^n$ ($n\ge 4$) into a   compact smooth Riemannian manifold $(N,h)$
 without boundary.  
 
 We first recall the necessary definitions. By embedding  $N$  isometrically into an Euclidean space $\mathbb{R}^m$
	for sufficiently large $m$, weakly  biharmonic  maps are defined as critical points with respect to
	 outer variations of the  second order energy functional
	$$
	E(u,\Om)=\int_{\Om}|\Delta u|^2dx \qquad \text{for }u\in H^{2}(\Om, N),
	$$
	where  $\Omega\subset\R^n$ is a bounded smooth domain, and
	$$H^{2}(\Om, N)=\Big\{u\in H^{2}(\Om, \R^m): \ u(x)\in N\text{ for a.e. }x\in \Om \Big\}.$$
	Note that the notion of weakly biharmonic maps may depend on the isometric embedding of $(N,h)$ into $\mathbb{R}^m$. 	

	The Euler-Lagrange equation for a weakly biharmonic map $u\in H^{2}(\Om, N)$ is given by
	$$
	\Delta^2u-\Delta\left(A(u)(\nabla u,\nabla u)\right)-2\nabla\cdot \langle \Delta u, \nabla P(u) \rangle+\langle \Delta (P(u)), \Delta u\rangle=0,
	$$
	where $A(\cdot)(\cdot,\cdot)$ is the second fundamental form of $N\hookrightarrow \mathbb{R}^m$,
	and $P(y):\R^m\to T_yN$ is the orthogonal projection of $\mathbb{R}^m$ into the tangent space $T_yN$ of $N$ at  $y\in N$. When $N=\mathbb{S}^{m-1}\subset \mathbb{R}^m$, the above equation reduces to
	\begin{equation}\label{eq:weakly bh into sphere}
		\Delta^2u=-\left(|\Delta u|^2+\Delta (|\nabla u|^2)+2\nabla u\cdot \nabla \Delta u\right)u.
	\end{equation}

From an analytic point of view,  dimension four is critical for the  equation \eqref{eq:weakly bh into sphere} in the sense that every weakly biharmonic map is smooth in dimension four, while singularity  appears in general in dimensions $n\ge 5$. Thus we say that the dimensions $n\ge 5$ are supercritical for biharmonic maps. Our aim in this paper is to study energy identity for stationary biharmonic maps in supercritical dimensions.
	\begin{definition}[Stationary biharmonic maps]\label{def:stationary bi-HM}
		A map $u\in H^{2}(\Om,N)$ is called a stationary biharmonic map if it is weakly biharmonic and if, in addition, it is a critical point of $E$ with respect to the inner variations, i.e. for all $\xi\in C_0^\infty(\Om,\R^n)$, it holds
		\begin{equation}\label{eq:def for stationary biHM}
			\frac{d}{dt}\Big|_{t=0}E(u\circ (id+t\xi),\Om)=0.
		\end{equation}
	\end{definition}
	
By direct calculations, \eqref{eq:def for stationary biHM} implies that	
	\begin{equation}\label{eq:id for stationary biHM}
		\int_\Om|\De u|^{2}{\rm div}\xi=\int_\Om\De u\cdot\Big(4\sum_{k=1}^n\na\pa_{k}u\cdot\na\xi^{k}+2\sum_{k=1}^n\De\xi^k\cdot\pa_ku\Big).
		\end{equation}
	A crucial property of stationary biharmonic maps is the following monotonicity formula for normalized energy (see \cite{Chang-W-Y-1999}): For any $x\in\Om$
	and $0<r\le R<{\rm{dist}}(x,\partial\Om)$,
	\[
	\Phi_{u}(a,r)-\Phi_{u}(a,\rho)=4\int_{B_{r}(a)\backslash B_{\rho}(a)}\left(\frac{|D\pa_{X}u|^{2}}{|x-a|^{n-2}}+(n-2)\frac{|\pa_{X}u|^{2}}{|x-a|^{n}}\right)dx,
	\]
	where $X=x-a$ and
	\[
	\Phi_{u}(a,r)=r^{4-n}\int_{B_{r}(a)}|\De u|^{2}+r^{3-n}\int_{\pa B_{r}(a)}\left(\pa_{X}|Du|^{2}+4|Du|^{2}-4\frac{|\pa_{X}u|^{2}}{|x-a|^{2}}\right).
	\]
	
	 The regularity theory is one of the central topics for biharmonic maps.	Chang-Wang-Yang \cite{Chang-W-Y-1999} initiated the study of regularity theory of sphere-valued (i.e., $N=\S^k$) weakly biharmonic maps and established their smoothness
	when $n=4$. They further proved that every stationary biharmonic map is smooth away from the singular set  $\sing(u)$ of $u$ which is defined as
	\[
	\sing(u):=\Big\{x\in \Omega: u \text{ is not continuous at any neighborhood of }x\Big\}.
	\] and proved that $\cal{H}^{n-4}(\sing(u))=0$ if $n\geq 5$; see also the independent work of P. Strzelecki \cite{Strzelecki-2003} using a different method.
	Shortly after, the second author of the present paper developed a regularity theory of  biharmonic maps into
	general compact Riemannian manifolds in a series of papers \cite{Wang-2004-CV, Wang-2004-MZ, Wang-2004-CPAM} via the method of Coulomb moving frames. In the critical dimension $n=4$, motivated by the celebrated result of Rivi\`ere \cite{Riviere-2007}, Lamm and Rivi\`ere \cite{Lamm-Riviere-2008} introduced a class of more general fourth order critical elliptic systems, including biharmonic maps, and proved the continuity of weak solutions. Built upon the techniques of \cite{Lamm-Riviere-2008}, Guo and Xiang  \cite{Guo-Xiang-2019-Boundary} derived the H\"older continuity of weak solutions, while Guo, Xiang and Zheng \cite{Guo-Xiang-Zheng-2021-CVPDE} established the sharp $L^p$-regularity theory; for the fourth order Lamm-Rivi\`ere system; see also \cite{Struwe-2008} for an alternative approach on the regularity theory of weakly biharmonic maps and related fourth order elliptic system, and Guo-Wang-Xiang \cite{Guo-Wang-Xiang-2023-CVPDE} further established an optimal partial $L^p$ regularity theory for a nonlinear fourth order elliptic system introduced by Struwe \cite{Struwe-2008} in dimensions $n\ge 5$.
	
	To further study the regularity theory of biharmonic maps in supercritical dimensions $n\geq 5$,  Scheven \cite{Scheven-2008-ACV} developed a defect measure theory for stationary biharmonic maps, extending the seminal result of Lin \cite{Lin-1999-Annals} for harmonic maps. More precisely, among other results, Scheven proved that, if $u_{k}\in H^{2}(\Om, N)$ is a sequence of stationary biharmonic mappings and
	\[
	u_{k}\wto u\qquad\text{weakly but not strongly in }H^{2}(\Om).
	\]
	Then $u$ is a weakly biharmonic map and
	\[
	|\De u_{k}|^{2}dx\wto|\De u|^{2}dx+\nu\qquad\text{ as weak convergence of Radon measures}
	\]
	with
	\begin{equation}\label{eq: density function}
		\nu=\Theta_{\nu}^{n-4}(x){\cal H}^{n-4}\LLcorner\Sigma
	\end{equation}
	for some countably $(n-4)$-rectifiable set $\Si\subset \Om$ and some density function $\Theta_{\nu}^{n-4}(x)$. Based on the results of Scheven \cite{Scheven-2008-ACV}, very recently, Guo-Jiang-Xiang-Zheng \cite{Guo-J-Xiang-Zheng-2025} obtained the optimal higher regularity for biharmonic maps via the powerful method of quantitative stratification of Naber-Valtorta \cite{Naber-Val-2017-AnnMath}.
	
To further understand the regularity theory of stationary biharmonic maps, a natural question then is to understand the density function $\Theta_{\nu}^{n-4}(x)$	as defined in \eqref{eq: density function}, which is also called the energy quantization problem in the literature.
	
	Such problems were first studied in the literature by Sacks and Uhlenbeck in the seminar work \cite{Sacks-Uhlenbeck-1981} on harmonic maps in the critical  dimension two.  They discovered that the loss of compactness of weakly convegent sequence of harmonic maps arises from the creation of bubbles, which are nontrivial harmonic maps from $\mathbb{S}^2$ to $N$, near each energy concentration point. Subsequently, Parker \cite{Parker1996}, Ding-Tian \cite{DingTian1995}, Qing-Tian \cite{QingTian1997} and Lin-Wang \cite{LinWang1998} established the energy identity that accounts for the loss of energy by the sum of energies of finitely many bubbles. It was until 2002 that Lin and Rivi\`ere \cite{Lin-Riviere-2002} established the first energy identity for stationary harmonic maps in supercritical dimensions ($n\ge 3$).

Similarly, the energy quantization for biharmonic maps was first studied in the critical dimension $n=4$, where weakly biharmonic maps are smooth as aforementioned and thus are stationary. More importantly, in this case the energy concentration set $\Sigma$ is a locally finite set by \cite{Scheven-2008-ACV} and energy indentity for statinoary biharmonic maps were established by Wang-Zheng \cite{Wang-Zheng-2012,Wang-Zheng-2013} and Hornung-Moser \cite{Hornung-Moser-2012-APDE}. See also  Liu-Yin \cite{Liu-Yin-2016-MZ} for an alternate proof,	Chen-Zhu \cite{Chen-Zhu-2023-SCM}  for related result on Riemannian 4-manifold, and Laurain-Rivi\`ere \cite{l-r-4} and Guo-Qi-Sun-Wang \cite{Guo-Qi-Sun-Wang-2026-AMS} for more general fourth order linear system.

	
The characterization of $\Theta_{\nu}^{n-4}(x)$  becomes substantially  more difficult in the supercritical dimensions, due to the  basic fact that the energy concentration set $\Sigma$ has higher Hausdorff dimension $\dim_{H}\Sigma=n-4\ge 1$. For harmonic maps, Lin and Rivi\`ere \cite{Lin-Riviere-2002} made the first breakthrough on the energy quantization problem for sphere-valued target $N=\mathbb{S}^{m-1}$ (see also \cite{zhang-zhu-2026} for an extension of this to homogeneous manifolds $N$). Very recently, in another significant breakthrough, Naber and Valtorta \cite{Naber-V-24-arXiv} succeeded in establishing the energy identity for stationary harmonic maps into general closed manifolds $N$ via a completely different new approach. However, due to the fourth order nature of biharmonic equations, it seems very challenging at the moment how to adapt their approach to biharmonic maps into general closed manifolds.
	Nevertheless, for $N=\mathbb{S}^{m-1}$, we are able to extend the seminal work of Lin-Rivi\`ere \cite{Lin-Riviere-2002} to stationary biharmonic mappings. 
	For simplicity, we will consider $\Om=B_1^n$, the unit ball in $\R^n$.
	\begin{theorem}\label{thm: main results} Suppose $u_{k}\in H^{2}(B_1^{n},\S^{m-1})$ $(n\ge 4)$
		is a sequence of sphere-valued stationary biharmonic mappings satisfying
		\[
		u_{k}\wto u\qquad\text{weakly but not strongly in }H^{2}(B_1^{n})
		\]
		and
		\[
		|\De u_{k}|^{2}dx\wto|\De u|^{2}dx+\nu\qquad\text{as weak convergence of Radon measures on $B_1^n$}
		\]
		with
		\[
		\nu=\Theta_{\nu}^{n-4}(x){\cal H}^{n-4}\LLcorner\Sigma
		\]
		for some countably $(n-4)$-rectifiable set $\Si\subset B_1^{n}$. Then  for $\mathcal{H}^{n-4}$ a.e. $x\in\Sigma$,
		there exists an integer $l_x$ such that
		\[
		\Theta_{\nu}^{n-4}(x)=\sum_{i=1}^{l_{x}}E(\phi_{i},\R^{4}),
		\]
		where each $\phi_{i}:\R^{4}\to\S^{m-1}$ is a nontrivial smooth biharmonic
		map.
	\end{theorem}
	
	Alternatively, one can also reformulate the above result by using $P_{\S^4}$-biharmonic maps $\phi\colon \S^4\to \S^{m-1}$. Recall that a map $\phi\in C^\infty(\S^4,\S^{m-1})$ is called a $P_{\S^4}$-biharmonic map if it is the critical point of the Paneitz energy functional
	\[
	E(v,\S^4):=\int_{\S^4}\big(|\Delta_{\S^4} v|^2+|\nabla_{\S^4}v|^2\big)\,d\HH^4,
	\]
	where $\Delta_{\S^4}$ and $\nabla_{\S^4}$ denote the Laplace and gradient operator with respect to the standard metric on $\S^4$. Let $\Pi\colon \S^4\to \R^4$ be the stereographic projection from the north pole. Then a smooth map $v\colon \S^4\to \S^{m-1}$ is $P_{\S^4}$-biharmonic if and only if $\bar{v}=v\circ \Pi^{-1}\colon \R^4\to \S^{m-1}$ is biharmonic. Moreover, $E(v,\S^4)=E(\bar{v},\R^4)$.  Thus in Theorem \ref{thm: main results}, we may equivalently write
	\[
	\Theta_{\nu}^{n-4}(x)=\sum_{i=1}^{l_{x}}E(\omega_{i},\S^{4})\qquad\text{for }{\cal H}^{n-4}\text{-a.e. }x\in\Si,
	\]
	where each $\omega_{i}:\S^{4}\to\S^{m-1}$ is a nontrivial $P_{\S^4}$-biharmonic map.

	The general strategy for the proof of Theorem \ref{thm: main results} is similar to that used by Lin-Rivi\'ere \cite{Lin-Riviere-2002} for harmonic maps and it consists of the following three main steps:
	\begin{enumerate}
		\item  A defect measure theory for stationary biharmonic maps so that we may construct the first nontrivial bubble.
		
		\item A key reduction lemma so that we may reduce the energy identity from supercritical dimension to the critical dimension.
		
		\item  $L^{(2,\infty)}$-estimate on the neck domain so that we may apply the $L^{(2,1)}-L^{(2,\infty)}$ duality to conclude vanishing energy on neck domains.
	\end{enumerate}
	To follow the strategy of Lin-Rivi\'ere \cite{Lin-Riviere-2002}, in the above steps, we need a higher order $L^{(2,1)}$-regularity theory for biharmonic maps, which seems to be only available for  symmetric target  manifolds such as $\mathbb{S}^{m-1}$,
	due to the fact that  in our case the  biharmonic equation \eqref{eq:weakly bh into sphere} can be written as the following conservation law
	\begin{equation}\label{eq:conservation law}
		\Delta(\nabla\cdot (\nabla u\wedge u))-2\nabla\cdot (\Delta u\wedge \nabla u)=0,
	\end{equation}
	 where $\nabla\cdot (\nabla u\wedge u)=\sum\limits_{i=1}^n\partial_i(\partial_iu\wedge u)$ and $\nabla\cdot (\Delta u\wedge \nabla u)=\sum\limits_{i=1}^n\partial_i(\Delta u\wedge \partial_iu)$.
	We would like to mention that a conservation law similar to \eqref{eq:conservation law} remains open for general target manifolds in supercritical dimensions; for critical dimension $n=4$, see Lamm-Rivi\`ere \cite{Lamm-Riviere-2008}.
	
	In our case, the first step was already established by Scheven in \cite{Scheven-2008-ACV}.  For later steps, we shall follow the strategy of Lin-Rivi\'ere \cite{Lin-Riviere-2002}, but make  suitable modifications  in order to carry out the entire blow-up process. In step 2, we provide a fully detailed proof of the reduction lemma, which was  only sketched by \cite{Lamm-Riviere-2008}. In the  last step, we  devote some effort to simplify the $L^{(2,\infty)}$-estimate on $\Delta_{X_2}u_i(X_1^i,X_2)$; see Lemma \ref{lem: small endpoint} and Proposition \ref{prop: key-estimate-1} below. In particular, we provide a direct proof of the key estimate in Proposition \ref{prop: key-estimate-1},
 instead of performing twice the conformal change of coordinates  as in \cite{Lin-Riviere-2002}. Indeed, if one tries to use a similar change of coordinates for biharmonic mappings,
then the natural change of time scale (corresponding to $t_i$ in \cite{Lin-Riviere-2002}) would be a power function, which fails to be conformal in dimension four.
This would  induce some additional difficulties, namely,  the equation in new coordinate fails to be biharmonic
and thus one  may not be able to apply the $\varepsilon$-regularity for  biharmonic mappings to the new equation to establish  the necessary estimates.


	\section{Proof of energy identity theorem}
	
	This section is devoted to the proof of Theorem \ref{thm: main results}. We divide it into three subsections.
	\subsection{Construction of the first bubble}
	
	We will outline the blow-up procedure to derive the first bubble
	by Scheven \cite[Section 3]{Scheven-2008-ACV}.
	
	 Under the same assumptions as in Theorem \ref{thm: main results}, after selecting a generic point $x_0$ in the energy concentration set $\Sigma$
	(such a point exists for ${\cal H}^{n-4}$ -a.e. $x\in \Sigma$)
	and considering the rescaled sequence $u_k(x_0+r_k\cdot)$ at $x_0$ with scalings $r_k\to 0$, we may assume, without loss of generality,
	that $x_0=0$ and  still denote the rescaled sequence of stationary biharmonic
	maps  as $u_{k}\in H^{2}(B_{2}^{n},N)$ with uniformly bounded  energies, that is,
	$\sup_{k\ge1}\|u_{k}\|_{H^{2}(B_{2}^{n})}\le\La$, such that
	\[
	u_{k}\wto c\quad\text{weakly in }H^{2}(B_{2}^{n})\text{ and strongly in }H^{1}(B_{2}^{n}),
	\]
	\[
	u_{k}\to c\quad\text{ in }C_{\loc}^{2}(B_{2}^{n}\backslash P)
	\]
	 for some constant $c\in \mathbb{S}^{m-1}$,
	and
	\[
	|\De u_{k}|^{2}dx\wto\Theta^{n-4}_{\nu}(0){\cal H}^{n-4}\LLcorner P
	\]
	 as weak convergence of Radon measures, where
	\[
	P=B_{2}^{n-4}\times\{0\}.
	\]
	By \cite[Lemma 2.4]{Scheven-2008-ACV},
	we can further assume that
	\begin{equation}
		\sup_{0<r<1,x\in B_{1}^{n}}r^{4-n}\int_{B_{r}^{n}}(|\De u_{k}|^{2}+r^{-2}|\na u_{k}|^{2})dx\le C,\quad\forall\,k\ge1.\label{eq: uniform morrey norm}
	\end{equation}
	for some universal constant $C=C(n,\La)>0$.
	
	To make the presentation more clear, we write $x=(X_1, X_2)$, where  $X_{1}=(x_{1},\cdots,x_{n-4})$ and $X_{2}=(x_{n-3},\cdots,x_{n})$.
	It follows from \cite[Section 3]{Scheven-2008-ACV} that there exist
	points $(X_{1}^{i},X_{2}^{i})\in B_1^{n-4}\times B_1^{4}$ and positive
	constants $\de_{i}\to0$  such that the following properties hold:
	\begin{itemize}
	\item[(i)] There holds
	\begin{equation}
		\sum_{j=1}^{n-4}\int_{B_1^{n}}|\na\pa_{x_{j}}u_{k}|^{2}dx+\int_{B_1^{n}}|\na u_{k}|^{2}dx\to0.\label{eq: X-0}
	\end{equation}
 The fact that the second term converges to zero follows from the compactness embedding $H^2(B^n)\hookrightarrow H^1(B^n)$ and our assumption.	Moreover, by \cite[Section 3]{Scheven-2008-ACV},
	\begin{equation}
		u_{k}\text{ is smooth in a neighborhood of }(X_{1}^{k},X_{2})\text{ for all }X_{2}\in B_1^{4}\label{eq: X-1}
	\end{equation}
	and, as $k\to\wq$, there holds
	\begin{equation}
		\sup_{0<r<1/2}r^{4-n}\int_{B_{r}^{n-4}(X_{1}^{k})\times B_{1}^{4}}\Big(\sum_{j=1}^{n-4}|\na\pa_{x_{j}}u_{k}|^{2}+r^{-2} |\na u_{k}|^{2}\Big)\,dX_1dX_2\to0\label{eq: X-2}
	\end{equation}
	
	\item[(ii)] There exist $\de_{k}\to0$ and $X_{2}^{k}\to0$ such that
	\begin{equation}
		\begin{aligned} & \quad\de_{k}^{4-n}\int_{B_{\de_{k}}^{n-4}(X_{1}^{k})\times B_{\de_{k}}^{4}(X_{2}^{k})}\left(|\De u_{k}|^{2}+\de_{k}^{-2}|\na u_{k}|^{2}\right)\,dX_1dX_2\\
			& =\max_{X_{2}\in B_1^{4}}\de_{k}^{4-n}\int_{B_{\de_{k}}^{n-4}(X_{1}^{k})\times B_{\de_{k}}^{4}(X_{2})}\left(|\De u_{k}|^{2}+\de_{k}^{-2}|\na u_{k}|^{2}\right)\,dX_1dX_2=\frac{\ep_{0}}{c(n)}
		\end{aligned}
		\label{eq: X-1-1}
	\end{equation}
	for some large constants $c(n)>0,$ whose value will be determined later.  Now, we define the blowing-up sequence
	\[
	v_{k}(y)=u_{k}\big((X_1^{k}, X_2^k)+\de_{k}y\big), \ y\in \delta_k^{-1}(B_1^n\setminus (X_1^k,X_2^k))
	\]
	there holds
	\begin{align*}
	\int_{B_{1}^{n-4}\times B_{1}^{4}}\left(|\De v_{k}|^{2}+|\na v_{k}|^{2}\right)
	&=\max_{X_{2}\in  \delta_k^{-1}(B_1^4\setminus\{X_2^k\})}\int_{B_{1}^{n-4}\times B_{1}^{4}(X_{2})}\left(|\De v_{k}|^{2}+|\na v_{k}|^{2}\right)\\
	&=\frac{\ep_{0}}{c(n)}>0.
	\end{align*}
	
	\item[(iii)] Up to a subsequence, there exists a nontrivial smooth  map
	$v(X_{1},X_{2})=v(X_{2})\in C^{\wq}\cap \dot{H}^{2}(\R^{n},\S^{k})$ such
	that
	\[
	v_{k}\to v\quad\text{in }C_{\loc}^{2}(\R^{n}).
	\]
	Moreover, $v:\R^{4}\to \mathbb{S}^{m-1}$ is a smooth biharmonic map with finite energy.
	\end{itemize}
	
\subsection{Reduction to dimension $4$}	
    First, we briefly recall the definition of Lorentz-Sobolev spaces. 	For a measurable function $f\colon \Om\to\R$, denote by $\de_{f}(t)=|\{x\in\Om:|f(x)|>t\}|$
	its distributional function and by $f^{\ast}(t)=\inf\{s>0:\de_{f}(s)\le t\}$,
	$t\ge0$, the nonincreasing rearrangement of $|f|$. Define
	\begin{eqnarray*}
		f^{\ast\ast}(t)\equiv\frac{1}{t}\int_{0}^{t}f^{\ast}(s)\D s, &  & t>0.
	\end{eqnarray*}
	The Lorentz space $L^{(p,q)}(\Om)$ ($1<p<\wq,1\le q\le\wq$)
	is the space of measurable functions $f:\Om\to\R$ such that
	\[
	\|f\|_{L^{(p,q)}(\Om)}\equiv\begin{cases}\displaystyle
		\left(\int_{0}^{\wq}(t^{1/p}f^{\ast\ast}(t))^{q}\frac{\D t}{t}\right)^{1/q}, & \text{if }1\le q<\wq,\\
		\displaystyle\sup_{t>0}t^{1/p}f^{\ast\ast}(t) & \text{if }q=\wq
	\end{cases}
	\]
	is finite.	
	
	The first order \emph{Lorentz-Sobolev space} $W^{1,(p,q)}(\Om)$ for $1<p<\wq,1\le q\le\wq$
	consists of functions $f\in L^{(p,q)}(\Om)$  with weak gradient $\na f\in L^{(p,q)}(\Om)$.
	A natural norm for a Lorentz-Sobolev function $f\in W^{1,(p,q)}(\Om)$ is defined  by
	\[
	\|f\|_{W^{1,(p,q)}(\Om)}=\Big(\|f\|_{L^{(p,q)}(\Om)}^{p}+\|\na f\|_{L^{(p,q)}(\Om)}^{p}\Big)^{1/p}.
	\]
	Higher order Lorentz-Sobolev spaces can be defined analogously.	
	
	Let $u\in H^2(\Omega,\mathbb{S}^{m-1})$ be a weakly biharmonic map such that  the conservation law \eqref{eq:conservation law} is satisfied. Since $u\in H^2(\Omega,\mathbb{S}^{m-1})$, for any ball $B\subset\subset \Omega$, $\Delta u\in L^2(B)$, $\nabla u\in L^{(4,2)}(B)$ and so $\Delta u\wedge \nabla u\in L^{(4/3,1)}(B)$. It follows from equation \eqref{eq:conservation law} and elliptic regularity theory in Lorentz spaces that
	\begin{equation}\label{eq:higher regularity}
		\|u\|_{W^{3,(\frac{4}{3},1)}(B)}\leq C(n,m)\|u\|_{H^2(\Omega)}.
	\end{equation}
  Applying \eqref{eq:higher regularity} to $u_k$ and noticing $\|u_k\|_{H^2(\Omega)}\leq \Lambda$ for all $k$, we obtain
	$$\|u_{k}\|_{W^{3,(\frac{4}{3},1)}(Q_{2/3}^{n})}\le C(n,m)\Lambda,$$
	where $Q_r^n=B_r^{n-4}\times B_r^4$,	and  hence by Fubini's Theorem
	\begin{equation}
		\sum_{j=1}^{3}\left\Vert \na^{j}u_{k}(X_{1},\cdot)\right\Vert _{L^{(\frac{4}{j},1)}(B_{2/3}^{4})}\le C(n,m)\Lambda\label{eq: X-3}
	\end{equation}
	 holds for all $X_{1}\in E_{k}\subset B_{2/3}^{n-4}$,
	 with ${\cal H}^{n-4}(E_{k})\ge0.99{\cal H}^{n-4}(B_{2/3}^{n-4})$,
	for all sufficiently large $k\ge1$.  Furthermore, we may find a subset
	$F_{k}\subset B_{2/3}^{n-4}$, with ${\cal H}^{n-4}(F_{k})\ge0.99{\cal H}^{n-4}(B_{2/3}^{n-4})$,
	such that for all $X_{1}\in F_{k}$,  both (\ref{eq: X-1})
	and (\ref{eq: X-2}) hold.
	
	\begin{lemma}[Reduction Lemma]\label{lem: main lemma}  For ${\cal H}^{n-4}$-a.e.
		$X_{1}^{k}\in E_{k}\cap F_{k}$,  there holds
		\[
		\lim_{k\to\wq}\int_{B_{1/2}^{4}(0)}\left|\De u_{k}(X_{1}^{k},X_{2})\right|^{2}\,dX_{2}=\Theta_{\nu}^{n-4}(0).
		\]
	\end{lemma}
	\begin{proof}
		For $0<\de\ll1$, let $\xi\in C_{0}^{\wq}(\R^{n})$ with ${\rm supp}(\xi)\subset Q_{\de}^n$
		 and for $a\in Q_{1-\de}^n$, define
		\[
		F_{k}(a)=\int_{Q_{1}^n}|\De u_{k}(x+a)|^{2}\xi(x)\,dx=\int_{Q_{1}^n}|\De u_{k}(x)|^{2}\xi(x-a)\,dx.
		\]
		Recall that the stationarity identity \eqref{eq:id for stationary biHM} of $u_{k}$ implies
		\[
		\int|\De u_{k}|^{2}{\rm div}\psi=\int\De u_{k}\cdot\left(\sum_{j=1}^{n}4\na\pa_{j}u_{k}\cdot\na\psi^{j}+\sum_{j=1}^{n}2\pa_{j}u_{k}\De\psi^{j}\right)
		\]
		for all $\psi\in C_{0}^{\wq}(Q_{1}^n,\R^{n})$. Hence, by taking $\psi=\xi e_{j}$,
		we deduce
		\[
		\pa_{a_{j}}F_{k}(a)=-\int_{Q_{1}^n}\De u_{k}(x+a)\cdot\left(4\na\pa_{j}u_{k}(x+a)\cdot\na\xi+2\pa_{j}u_{k}(x+a)\De\xi\right).
		\]
		
		Now we choose $\xi(X_{1},X_{2})=\xi_{1}^{\eta}(X_{1})\xi_{2}(X_{2})$,
		with $\xi_{1}^{\eta}=\eta^{4-n}\xi_{1}(\cdot/\eta)$ for some $0\le\xi_{1}\in C_{0}^{\wq}(B_1^{n-4})$
		and $\eta>0$ sufficiently small, and $0\le\xi_{2}\in C_{0}^{\wq}(B_1^{4})$.
		Define  function
		\[
		F_{k}^{\eta}(a)=\int_{Q_{1}}|\De u_{k}(X_{1}+a,X_{2})|^{2}\xi_{1}^{\eta}(X_{1})\xi_{2}(X_{2})dx\qquad \text{for } a\in B_{1-\eta}^{n-4}.
		\]
		It follows that, for $1\le j\le n-4$,
		\[
		\begin{aligned}\pa_{a_{j}}F_{k}^{\eta}(a) & =-\int_{Q_{1}^n}\De u_{k}\cdot\Big(\sum_{l=1}^{n}4\pa_{lj}u_{k}\cdot\pa_{l}\xi(X_{1}-a,X_{2})+2\pa_{j}u_{k}\De\xi(X_{1}-a,X_{2})\Big)\\
			& =\sum_{l=1}^{n-4}\pa_{a_{l}}\Big(4\int_{Q_{1}^n}\De u_{k}(X_{1}+a,X_{2})\cdot\pa_{lj}u_{k}(X_{1}+a,X_{2})\xi_{1}^{\eta}\xi_{2}\Big)\\
			& \qquad-\sum_{l=n-3}^{n}\int_{Q_{1}}\De u_{k}(X_{1}+a,X_{2})\cdot4\pa_{lj}u_{k}(X_{1}+a,X_{2})\xi_{1}^{\eta}\pa_{l}\xi_{2}\\
			& \qquad-\int_{Q_{1}^n}\De u_{k}(X_{1}+a,X_{2})\cdot2\pa_{j}u_{k}(X_{1}+a,X_{2})\De\xi\\
			& :={\rm div}_aG_{k}^{\eta}(a)+H_{k}^{\eta}(a),
		\end{aligned}
		\]
		where
		\[
		G_{k}^{\eta}(a)=\Big(4\int_{Q_{1}^n}\De u_{k}(X_{1}+a,X_{2})\cdot\pa_{lj}u_{k}(X_{1}+a,X_{2})\xi_{1}^{\eta}\xi_{2}\Big)_{1\le l\le n-4}
		\]
		and
		\[
		\begin{aligned}H_{k}^{\eta}(a) & =-\sum_{l=n-3}^{n}\int_{Q_{1}^n}\De u_{k}(X_{1}+a,X_{2})\cdot4\pa_{lj}u_{k}(X_{1}+a,X_{2})\xi_{1}^{\eta}\pa_{l}\xi_{2}\\
			& \qquad-\int_{Q_{1}^n}\De u_{k}(X_{1}+a,X_{2})\cdot2\pa_{j}u_{k}(X_{1}+a,X_{2})\De\xi.
		\end{aligned}
		\]
		By (\ref{eq: X-0}), there holds
		$$\|G_{k}^{\eta}\|_{L^{1}(B_{1-\eta}^{n-4})}+\|H_{k}^{\eta}\|_{L^{1}(B_{1-\eta}^{n-4})}\to0$$
		as $k\to\wq$ uniformly in $\eta$.
		Thus, by Allard's strong
		constancy Lemma (see e.g. \cite[Lemma 2.7]{Lin-1999-Annals}) we deduce that
		\begin{equation}\label{eq: X-1}
			\|F_{k}^{\eta}-c_{k}^{\eta}\|_{L_{\loc}^{1}(B_{1-\eta}^{n-4})}\to0\qquad\text{as }k\to\wq 
		\end{equation}
		for some constant $c_{k}^{\eta}$ and the above convergence is uniform with respect to  $\eta$. 
		
		On the other hand, it is straightforward to show that $F_{k}^{\eta}(a)\to F_{k}^{0}(a)$
		in $L_{\loc}^{1}(B^{n-4}_1)$ as $\eta\to0$, where
		\[
		F_{k}^{0}(a)=\int_{B_1^{4}}|\De u_{k}(a,X_{2})|^{2}\xi_{2}(X_{2})dX_{2}, \qquad  a\in B_1^{n-4}.
		\]
		Combining with \eqref{eq: X-1}, this implies that $\{c_{k}^{\eta}\}_{\eta>0}$ is a Cauchy sequence
		as $\eta\to0$ (for each $k\gg 1$). Hence there exist $\{c_{k}\}\subset\R$ such
		that
		\begin{equation}\label{eq:consequence of Allard strong constancy lemma}
			\lim_{k\to\wq}\int_{B_{r}^{n-4}(X_{1})}|F_{k}^{0}(a)-c_{k}|\,da=0
		\end{equation}
		for any $B_{r}^{n-4}(X_{1})\Subset B^{n-4}_1$.
		
		Next, we claim that for any $B_{r}^{n-4}(X_{1})\Subset B^{n-4}_1$,  it holds
		\begin{equation}\label{eq:key for reduction}
			\lim_{k\to\wq}\int_{B_{r}^{n-4}(X_{1})}F_{k}^{0}(a)\,da =\Theta_{\nu}^{n-4}(0){\cal H}^{n-4}(B_{r}^{n-4}).
		\end{equation}
		
		Indeed, for any $B_{r}^{n-4}(X_{1})\Subset B^{n-4}_1$, take $\xi_1^k\in C_0^\infty(B_{r}^{n-4}(X_{1}))$ such that
		 $\xi_1^j$ is monotonically increasing and converges to the constant function $1$, as $j\to\infty$, on $B_{r}^{n-4}(X_{1})$, and $\xi_{2}\equiv1$ on $B_{1/2}^{4}$.
		Let $$C_k^j=\int_{B_{r}^{n-4}(X_{1})}F_{k}^{0}(a)\xi_1^j(a)\,da.$$
		Then for each fixed $k$, $C_k^j$ is monotonically increasing with respect to $j$.
		Moreover, since $|\De u_{k}|^{2}dx\wto\Theta_{\nu}^{n-4}(0){\cal H}^{n-4}\LLcorner B_1^{n-4}\times\{0\}$, we have
		$$
		\begin{aligned}
			\lim_{k\to \infty}C_k^j&=\lim_{k\to \infty}\int_{B_{r}^{n-4}(X_{1})}F_{k}^{0}(a)\xi_1^j(a)da\\
			&=\lim_{k\to\wq}\int_{B_{r}^{n-4}(X_{1})\times B_1^{4}}|\De u_{k}(x)|^{2}\xi_1^j(a)\xi_{2}(X_{2})dx\\
			&=\Theta_{\nu}^{n-4}(0)\int_{B_{r}^{n-4}(X_{1})}\xi_1^j(a)\,d{\cal H}^{n-4}(a)=:C^j,
		\end{aligned}
		$$
		and the monotone convergence theorem gives
		$$
		\lim_{j\to \infty} C^j=\lim_{j\to \infty}\Theta_{\nu}^{n-4}(0)\int_{B_{r}^{n-4}(X_{1})\times\{0\}}\xi_1^j(a)d{\cal H}^{n-4}(a)=\Theta_{\nu}^{n-4}(0){\cal H}^{n-4}(B_{r}^{n-4}).
		$$
		Similarly, by the monotone convergence theorem, we have
		$$
		\lim_{j\to\infty}C_k^j=\lim_{j\to \infty}\int_{B_{r}^{n-4}(X_{1})}F_{k}^{0}(a)\xi_1^j(a)da=\int_{B_{r}^{n-4}(X_{1})}F_{k}^{0}(a)da=:C_k,
		$$
		and
		$$\lim_{k\to \infty}C_k=(\lim_{k\to\infty}c_k){\cal H}^{n-4}(B_{r}^{n-4}).$$
		Set $d^j=\int_{B_{r}^{n-4}(X_{1})}\xi_1^j(a)da$.
		Then by \eqref{eq:consequence of Allard strong constancy lemma} and our choice of $\xi_1^j$, we have
		\[
		\begin{aligned}
			\lim_{k\to\infty}\sup_{j}|C_k^j-c_kd^j|&\leq
			\lim_{k\to\infty}\sup_{j}\int_{B_{r}^{n-4}(X_{1})}\left|F_{k}^{0}(a)\xi_1^j(a)-c_k\xi_1^j(a)\right|\,da\\
			&\leq \lim_{k\to\infty}\int_{B_{r}^{n-4}(X_{1})}\left|F_{k}^{0}(a)-c_k\right|\,da=0.
		\end{aligned}
		\]
		This implies that as $k$ is large enough, $C_k^j$ is close to $c_kd^j$, uniformly in $j$. In particular, $C_k^j$ converges to $C^j$ as $k\to\infty$ uniformly in $j$. Applying Lemma \ref{lemma:simple analysis} below gives
		\[
		\lim_{k\to \infty}\lim_{j\to \infty}C_k^j=\lim_{j\to \infty}\lim_{k\to \infty}C_k^j.
		\]
		It follows
		$$
		\lim_{k\to\wq}\int_{B_{r}^{n-4}(X_{1})}F_{k}^{0}(a)da =\Theta_{\nu}^{n-4}(0){\cal H}^{n-4}(B_{r}^{n-4}).
		$$
		Hence, for any $B_{r}^{n-4}(X_{1})\Subset B^{n-4}$,
		\[
		\lim_{k\to\wq}\fint_{B_{r}^{n-4}(X_{1})}F_{k}^{0}(a)da=\Theta_{\nu}^{n-4}(0).
		\]
		On the other hand, note that
		\[
		\begin{aligned}\int_{B_{r}^{n-4}(X_{1})}F_{k}^{0}(a)\,da & =\int_{B_{r}^{n-4}(X_{1})\times\left(B_1^{4}\backslash B_{1/2}^{4}\right)}|\De u_{k}(a,X_{2})|^{2}\xi_{2}(X_{2})\,dadX_{2}\\
			& \quad+\int_{B_{r}^{n-4}(X_{1})\times B_{1/2}^{4}}|\De u_{k}(a,X_{2})|^{2}\,dadX_{2}.
		\end{aligned}
		\]
		The first integral in the right hand side vanishes as $k\to\wq$, since $\De u_{k}\to0$ uniformly
		on $B_{r}^{n-4}(X_{1})\times\left(B^{4}\backslash B_{1/2}^{4}\right)$.
		Hence
		\[
		\lim_{k\to\wq}\fint_{B_{r}^{n-4}(X_{1})}\int_{B_{1/2}^{4}}|\De u_{k}(X_{1},X_{2})|^{2}dX_{2}dX_{1}=\Theta_{\nu}^{n-4}(0).
		\]
		Since the above limit holds for all $B_{r}^{n-4}(X_{1})\Subset B_1^{n-4}$,
		the desired result follows.
	\end{proof}
	
	In the proof above, we used the following simple  fact which may be an exercise in some textbook.
	
	\begin{lemma}\label{lemma:simple analysis}
		Let $\{C_i^j\}_{i,j \in \mathbb{N}}$ be a sequence of nonnegative numbers such that for each fixed $i$, the sequence $C_i^j$ is monotonically increasing in $j$. Assume that
		\[
		C_i^j\overset{i}{\rightrightarrows} C^j, \quad \lim_{j\to\infty} C^j = \alpha,
		\]
		and
		\[
		\lim_{j\to\infty} C_i^j = C_i, \quad \lim_{i\to\infty} C_i = \beta.
		\]
		Then $\alpha = \beta$.
	\end{lemma}
	
	\begin{proof}
Let $\ep>0$. Take sufficiently large $i_1$ s.t. $|\be-C_{i}|<\ep$ for all $i\ge i_1$. Since  $ C_i^j\overset{i}{\rightrightarrows} C^j$ as $i\to \wq$, there exists $i_2\gg 1$ s.t. $|C_{i}^j-C^j|<\ep$ for $i\ge i_2$ and all $j$.  Let $i_0=\max\{i_1, i_2\}$. Since $\lim_{j\to\infty} C^j = \alpha$ and $\lim_{j\to\infty} C_i^j = C_i$, there exists $j_0\gg 1$ s.t. $|C_{i_0}-C_{i_0}^j|+|\alpha-C^j|<\ep$ for $j\ge j_0$.  Therefore, for $j\ge j_{0}$, we have
\[
|\beta-\alpha|\leq |\be-C^j|+|\alpha-C^j|
\le |\be- C_{i_0}|+\ep+|C_{i_0}^j-C^j|<3\ep.
\]
The proof is complete by sending $\ep$ to 0.
	\end{proof}

	\subsection{$L^{(2,\wq)}$-estimate of $\De_{X_{2}}u_{k}(X_{1}^{k},X_{2})$}
	With the reduction Lemma 2.1 at hand, we find that
	\[
	\begin{aligned}&\int_{B_{1/2}^{4}(X_{2}^{k})}\left|\De_{X_{2}}u_{k}(X_{1}^{i},X_{2})\right|^{2}\,dX_{2}\\
	&=\Big(\int_{B_{1/2}^{4}(X_{2}^{k})\backslash B_{\de_{k}R}^{4}(X_{2}^{k})}+\int_{B_{\de_{k}R}^{4}(X_{2}^{k})}\Big)\Big|\De_{X_{2}}u_{k}(X_{1}^{k},X_{2})\Big|^{2}\,dX_{2}\\
		& =\int_{A(R,k)}\left|\De_{X_{2}}u_{k}(X_{1}^{k},X_{2})\right|^{2}dX_{2}+\int_{B_{R}^{4}}|\De_{Y_{2}}v_{k}(0,Y_{2})|^{2}\,dY_2,
	\end{aligned}
	\]
	where $v_{k}(y)=u_{i}((X_1^{k}, X_2^k)+\de_{k}Y)$ and
	\[
	A(R,k)=B_{1/2}^{4}(X_{2}^{k})\backslash B_{\de_{k}R}^{4}(X_{2}^{k}).
	\]
	By a standard reduction procedure (see \cite{Lin-Riviere-2002}), we may assume there is only one bubble at $(0,0)$.
	In this case, since $v_{k}$ converges to a bubble $v\in H_{\loc}^{2}(\R^{4})\cap C_{\loc}^{2}(\R^{4})$,
	the energy identity would follow if we can prove that
	\begin{equation}\label{eq: X-4}
		\lim_{R\to\wq}\lim_{k\to\wq}\int_{A(R,k)}|\De_{X_{2}}u_{k}(X_{1}^{k},X_{2})|^{2}\,dX_{2}=0.
	\end{equation}
	
	 It turns out that we can show \eqref{eq: X-4} by applying the duality between the Lorentz spaces
	$L^{(2,1)}$ and $L^{(2,\wq)}$. For this,  we need to estimate
	the $L^{(2,\wq)}$-norm of\textbf{ $\De_{X_{2}}u_{k}(X_{1}^{k},X_{2})$}
	on $A(R,k)$.
	
	
	\begin{lemma}\label{lem: small endpoint} For any $R>1$ and $0<\la\ll1/2$
		such that $\de_{k}R/\la<1/2$, there holds,
		\[
		\lim_{k\to\wq}\int_{B_{1/2}^{4}(X_{2}^{k})\backslash B_{\la}^{4}(X_{2}^{k})}|\De_{X_{2}}u_{k}(X_{1}^{k},X_{2})|^{2}\,dX_{2}=0,
		\]
		and
		\[
		\lim_{R\to\wq}\lim_{k\to\wq}\int_{B_{\de_{k}R/\la}^{4}(X_{2}^{k})\backslash B_{\de_{k}R}^{4}(X_{2}^{k})}|\De_{X_{2}}u_{k}(X_{1}^{k},X_{2})|^{2}\,dX_{2}=0.
		\]
	\end{lemma}

	\begin{proof} The first term holds, since $|\De u_{k}|^{2}\to0$ uniformly on $B_{1}^{n}(0)\backslash B_{\la/2}^{4}(0)$
		and since $X_{2}^{i}\to0$. On the other hand, recall that the function
		$v_{k}(Y_{1},Y_{2})=u_{k}(X_{1}^{k}+\de_{k}Y_{1},X_{2}^{k}+\de_{k}Y_{2})$
		converges to $v\in W^{2,2}(\R^{4},\mathbb S^{m-1})$ in $C_{\loc}^{2}(\R^n)$.
		Thus
		\[
		\begin{aligned} & \int_{B_{\de_{k}R/\la}^{4}(X_{2}^{k})\backslash B_{\de_{k}R}^{4}(X_{2}^{k})}|\De_{X_{2}}u_{k}(X_{1}^{k},X_{2})|^{2}\,dX_{2}\\
			& \qquad=\int_{\{R\le|Y_{2}|\le R/\la\}}|\De_{Y_{2}}v_{k}(0,Y_{2})|^{2}\,dY_{2}\\
			& \qquad\to\int_{\{R\le|Y_{2}|\le R/\la\}}|\De v(Y_{2})|^{2}\,dY_{2},
		\end{aligned}
		\]
		as $k\to\wq$. Since $\nabla^2 v\in L^2(\R^{4})$, the second assertion
		holds.
	\end{proof}
	Now we prove the following key estimate.
	
	\begin{proposition} \label{prop: key-estimate-1} For any $\ep>0$,
		there exists $R_{0}$ sufficiently large such that
		\begin{equation}\label{no-neck-bubble}
		\limsup_{k\to\wq}\sup_{\rho\in(\de_{k}R,1/4)}I(\rho,u_{k},X^{k})\le\ep
		\end{equation}
		for $R\ge R_{0}$, where $X^k=(X_1^k,X_2^k)$ and
		\[
		I(\rho,u_{k},X^{k})=\rho^{4-n}\int_{B_{\rho}^{n-4}(X_{1}^{k})}\int_{B_{2\rho}^{4}(X_{2}^{k})\backslash B_{\rho}^{4}(X_{2}^{k})}\left(|\De u_{k}|^{2}+\rho^{-2}|\na u_{k}|^{2}\right)\,dx.
		\]
	\end{proposition}
	
	 Note that it follows from \eqref{no-neck-bubble}  that
	\[
	\lim_{R\to\wq}\limsup_{k\to\wq}\sup_{\rho\in(\de_{k}R,1/4)}I(\rho,u_{k},X^{k})=0.
	\]
	The limit with respect to $R$ exists due to the fact that the supremum
	is monotonically decreasing with respect to $R$.
	\begin{proof}
		Suppose, on the contrary, that there exists $\ep_{1}>0$, such that
		for any sufficiently large $R$, there exist $\rho_{k}\in(\de_{k}R,1/4)$
		 so that
		\begin{equation}
			I(\rho_{k},u_{k},X^{k})=\sup_{\rho\in(\de_{k}R,1/4)}I(\rho,u_{k},X^{k})\ge\ep_{1}.\label{eq: X-6}
		\end{equation}
		
		\noindent\emph{Claim.} There hold both $\rho_{k}\to0$ and $\de_{k}R/\rho_{k}\to0$
		as $k\to\wq$. Indeed, if $\rho_{k}\ge c>0$ as $k\to\wq$, then $$B_{2\rho_{k}}^{4}(X_{2}^{k})\backslash B_{\rho_{k}}^{4}(X_{2}^{k})\subset B^{4}(0)\backslash B_{c/2}^{4}(0),$$
		which  implies $(|\De u_{k}|^{2}+\rho_{k}^{-2}|\na u_{k}|^{2})$
		converges uniformly to $0$  on $B_{2\rho_{k}}^{4}(X_{2}^{k})\backslash B_{\rho_{k}}^{4}(X_{2}^{k})$.
		This implies that
		$$I(\rho_{k},u_{k},X^{k})=o_{k}(1)\rho_{k}^{4}=o_{k}(1)\to0$$
		as $k\to\wq$, which contradicts  to  (\ref{eq: X-6}).
		On the other hand, since $\rho_{k}\ge\de_{k}R$, if $\rho_{k}/(\de_{k}R)\to c\in[1,\wq)$,
		then  by considering $v_{k}(Y)=u_{k}(X^{k}+\de_{k}Y)$ we deduce
		
		\[
		\begin{aligned}\ep_{1} & \le\left(\frac{\rho_{k}}{\de_{k}}\right)^{4-n}\int_{B_{\rho_{k}/\de_{k}}^{n-4}}\int_{B_{2\rho_{k}/\de_{k}}^{4}\backslash B_{\rho_{k}/\de_{k}}^{4}}
		\Big(|\De v_{k}|^{2}+\left(\frac{\rho_{k}}{\de_{k}}\right)^{-2}|\na v_{k}|^{2}\Big)\,dY\\
			& \to(cR)^{4-n}\int_{B_{cR}^{n-4}}\int_{B_{2cR}^{4}\backslash B_{cR}^{4}}\Big(|\De v|^{2}+\left(cR\right)^{-2}|\na v|^{2}\Big),\,dY\\
			& =\om_{n-4}\int_{B_{2cR}^{4}\backslash B_{cR}^{4}}\Big(|\De v|^{2}+\left(cR\right)^{-2}|\na v|^{2}\Big), dY_{2},
		\end{aligned}
		\]
		which converges to zero as $R\to\wq$, since $\nabla v\in {H}^{1}(\R^{4})$.  This contradicts to (\ref{eq: X-6}) again.
		The claim is therefore proved.
		
		Now, define $w_{k}(Y)=u_{k}(X^{k}+\rho_{k}Y)$ for $|Y|\le1/2\rho_{k}\to\wq$.
		Then the estimate (\ref{eq: uniform morrey norm}) implies that,
		for any $r>0$,
		\begin{equation}
			\begin{aligned} & r^{4-n}\int_{B_{r}^{n}}\Big(|\De w_{k}|^{2}+r^{-2}|\na w_{k}|^{2}\Big)\,dY\\
				&=(\rho_{k}r)^{4-n}\int_{B_{\rho_{k}r}^{n}(X^{k})}\Big(|\De u_{k}|^{2}+(\rho_{k}r)^{-2}|\na u_{k}|^{2}\Big)\,dY\le C
			\end{aligned}
			\label{eq: X-7}
		\end{equation}
		for all $k\gg1$, i.e., $w_{k}$ is locally uniformly bounded in $W^{2,2}$;
		and (\ref{eq: X-6}) is equivalent to
		\begin{equation}
			\int_{B_1^{n-4}}\int_{B_{2}^{4}\backslash B_1^{4}}\left(|\De w_{k}|^{2}+|\na w_{k}|^{2}\right)dy\ge\ep_{1}\label{eq: normalization}
		\end{equation}
		for all $k\gg1$. Hence we may assume again, up to a subsequence that,
		$$w_{k}\wto w_{\wq} \ {\rm{in}}\ W_{\loc}^{2,2}(\R^{n})\ \ {\rm{ and \ }} \ \ |\De w_{k}|^{2}dy\wto|\De w_{\wq}|^{2}dy+\nu_{2}.$$
		 as weak convergence of Radon measures for some nonnegative defect measure
		$\nu_{2}$.
		
		If $\nu_{2}=0$, then we have $w_{i}\to w_{\wq}$ strongly in $W_{\loc}^{2,2}(\R^{n})$.
		By  (\ref{eq: X-2}), for any $r>0$ we have
		\[
		r^{4-n}\int_{B_{r}^{n-4}\times B_{1/\rho_{k}}^{4}}\sum_{j=1}^{n-4}|\na\pa_{j}w_{k}|^{2}=(\rho_{k}r)^{4-n}\int_{B_{\rho_{k}r}^{n-4}\times B_{1}^{4}}\sum_{j=1}^{n-4}|\na\pa_{j}u_{k}|^{2}\to0
		\]
		as $k\to\wq$. By  (\ref{eq: normalization}) we see that
		$w_{\wq}$ is a nontrivial map. Hence we conclude from these two facts and
		the estimate (\ref{eq: X-7}) that $w_{\wq}(X)=w_{\wq}(X_{2}):\R^{4}\to \mathbb{S}^{m-1}$
		is a nontrivial, smooth biharmonic mapping from $\R^{4}$ into $\S^{k}$
		with a bounded energy.
		
		If $\nu_{2}$ is nontrivial, then we
		may repeat the same blow up procedure to find another bubble.
		
		Both cases contradict to  the assumption that there is only one bubble.
		The proof is complete.
	\end{proof}
	Now we are ready to prove Theorem \ref{thm: main results}.
	\begin{proof}[Proof of Theorem \ref{thm: main results}]
		By Proposition \ref{prop: key-estimate-1} and the standard small energy regularity
		 theorem (see e.g. \cite{Scheven-2008-ACV,Wang-2004-CPAM}), for $\ep>0$ sufficiently small, we infer that
		\[
		|X_{2}-X_{2}^{k}|^{2}|\na^{2}u_{k}(X_{1}^{k},X_{2})|\le C\sqrt{\ep}
		\]
		on  $\{X_{1}^{i}\}\times \big(B_{1/4}^{4}\backslash B_{\delta_kR}^{4}\big)$. This
		implies that the set
		$$A_{\la}:=\{x\in A(R,k):|\na^{2}u_{k}(X_{1}^{k},X_{2})|>\la\}$$
		satisfies
		\[
		|X_{2}-X_{2}^{k}|\le\frac{C\sqrt{\ep}}{\sqrt{\la}}\qquad\forall\,X_{2}\in A_{\la}.
		\]
		That is, $A_{\la}\subset B^{4}(X_{2}^{k},C\ep/\sqrt{\la})$. Thus
		\[
		\|\na^{2}u_{k}\|_{L^{(2,\wq)}(\{X_{1}^{k}\}\times A(R,k))}=\sup_{\la>0}\la|A_{\la}|^{1/2}\le C\ep.
		\]
		This, together with the duality between the Lorentz spaces $L^{(2,1)}$ and $L^{(2,\infty)}$, implies that
		\[
		\int_{A(R,k)}|\De_{X_{2}}u_{k}(X_{1}^{k},X_{2})|^{2}\,dX_{2}\le\|\na^{2}u_{k}\|_{L^{(2,1)}(\{X_{1}^{k}\}\times A(R,k))}^2\|\na^{2}u_{k}\|_{L^{(2,\wq)}(\{X_{1}^{k}\}\times A(R,k))}^2\le C\ep^2.
		\]
		Since $\ep>0$ is arbitrary, \eqref{eq: X-4} follows and hence the proof is complete.
	\end{proof}
	
	\subsection{A concluding remark on Lie group target}
	Our argument in the above can be applied to even more general settings. For instance, suppose that the target manifold is a submanifold of the special orthogonal group $G=SO(m)$. It is known that its Lie algebra $\mathfrak{g}=so(m)$.
	For $\Om\subset\R^n$, let $U\in H^2(\Omega, G)$ be a biharmonic map.
	
	Let $\varphi\in C_0^\infty(\Omega,\mathfrak{g})$ be a test function. Then $\varphi^T+\varphi=0$, and
	\[U_t=U\exp(t\varphi)\colon \Omega\to G \]
	is a family of variations of $U$. Thus
	\[
	\begin{aligned}
	0&=\frac{d}{dt}\Big|_{t=0}\int_{\Omega}|\Delta U_t|^2dx=2\int_{\Omega}\left\langle \Delta U, \Delta\Big(\frac{d}{dt}|_{t=0}U_t\Big)\right\rangle dx\\
	&=2\int_{\Omega}\left\langle\Delta U, \Delta (U\varphi)\right\rangle dx
	=2\int_{\Omega}\left\langle\Delta U, \left(\Delta U\varphi+2\nabla U\cdot \nabla \varphi+U\Delta \varphi \right)\right\rangle dx\\
	&=2\int_{\Omega}\left\langle\Delta U, \big(2\nabla U\cdot \nabla\varphi +U\Delta \varphi \big)\right\rangle dx, 
	\end{aligned}
	\]
	where in the last equality we used the fact that $\var$ is anti-symmetric.
	Above, $\left\langle A, B\right\rangle=\sum_{i,j=1}^n A^i_j B^i_j$ denotes the inner product of $n\times n$ matrices. By integrating by parts  we find that $U$ satisfies the equation
	\begin{equation}\label{eq:Lie subgroup case1}
		\Delta\left(U^T\Delta U \right)=2\divv(\Delta U \nabla U)\quad\ {\rm{in}}\ \ \ \Omega.
	\end{equation}
	Since
	\[
	U^T\Delta U=\divv(U^T\nabla U)-\nabla U^T\cdot\nabla U,
	\]
	we obtain that
	\begin{equation}\label{eq:Lie subgroup case2}
		\Delta\divv(U^T\nabla U )=2\divv(\Delta U \nabla U)+\divv(\nabla U^T\cdot\nabla U) \quad\ {\rm{in}}\ \ \ \Omega.
	\end{equation}
	Then it follows from \eqref{eq:Lie subgroup case1} and the standard elliptic regularity theory in Lorentz spaces that $U\in W^{3,(\frac{4}{3},1)}_{\rm{loc}}(\Om)$,
	and
	\[
	\big\|U\big\|_{W^{3,(\frac{4}{3},1)}(\Om')}\leq C(n,m, {\rm{dist}}(\Om',\partial\Om))\big\|U\big\|_{H^{2}(\Omega)}, \ \forall\Om'\Subset\Om.
	\]
	Thus
	$$\big\|\nabla^2U\big\|_{L^{(2,1)}(\Om')}\leq C(n,m, {\rm{dist}}(\Om',\partial\Om))\big\|U\big\|_{H^{2}(\Omega)},  \ \forall\Om'\Subset\Om.$$
	Then, by the same argument as in the proof of Theorem \ref{thm: main results}, the conclusions of Theorem \ref{thm: main results}
	remain to hold for a sequence of weakly convergent biharmonic mappings $\{U_k\}\subset H^2(\Omega, G)$.
	


\begin{thebibliography}{10}
		\bibitem{DeLellis-book} \textsc{C. De Lellis}, \emph{Rectifiable sets, densities and tangent measures.} Zur. Lect. Adv. Math. European Mathematical Society (EMS), Z\"urich, 2008.
		
		\bibitem{Chang-W-Y-1999} \textsc{S. Y. A. Chang, L. Wang and P. C. Yang,}
		\emph{A regularity theory of biharmonic maps.} Commun. Pure Appl.
		Math. \textbf{52}(9) (1999), 1113-1137.
		
		\bibitem{Chen-Zhu-2023-SCM}
		\textsc{Y. Chen and M. Zhu}, \emph{Bubbling analysis for extrinsic biharmonic maps from general Riemannian 4-manifolds}. Sci. China Math. \textbf{66} (2023), no. 3, 581-600.
		
		
		\bibitem{Evans-1991-ARMA} \textsc{C. L. Evans,} \emph{Partial regularity for stationary harmonic maps into spheres.} Arch. Rat. Mech. Anal. \textbf{116} (1991), 101-163.
		
		\bibitem{DingTian1995} \textsc{W. Y. Ding and G. Tian}, \emph{Energy identity for a class of approximate harmonic maps from surfaces}. Comm. Anal. Geom. \textbf{3} (1995), 543-554.
		
		\bibitem{Guo-J-Xiang-Zheng-2025}
		\textsc{C.-Y. Guo, G.-C. Jiang, C.-L. Xiang and G.-F. Zheng,} \emph{Optimal higher regularity for biharmonic maps via quantitative stratification}. Published online at Peking Math. J., 2025.
		
		\bibitem{Guo-Qi-Sun-Wang-2026-AMS}
		\textsc{C.-Y. Guo, W.-J. Qi, Z.-M. Sun and C.Y. Wang}, \emph{The Lamm-Rivi\`ere system II: energy identity}. To appear in Acta Math. Sci. Ser. B (Engl. Ed.), 2026.
		
		\bibitem{Guo-Wang-Xiang-2023-CVPDE} \textsc{C.-Y. Guo, C.Y. Wang and C.-L. Xiang}, \emph{ $L^p$-regularity  for fourth order elliptic systems with antisymmetric potentials in higher dimensions.}  Calc. Var. Partial Differential Equations \textbf{62} (2023), no. 1, Paper No. 31.
		
		\bibitem{Guo-Xiang-2019-Boundary}
		\textsc{C.-Y. Guo and C.-L. Xiang}, \emph{Regularity of solutions for a fourth order linear system via conservation law}.  J. Lond. Math. Soc. (2) \textbf{101} (2020), no. 3, 907-922.
		
		\bibitem{Guo-Xiang-Zheng-2021-CVPDE}
		\textsc{C.-Y. Guo, C.-L. Xiang and G.-F. Zheng,}\emph{The Lamm-Rivi\`ere system I:  $L^p$  regularity theory.} Calc.
		Var. Partial Differ. Equ. \textbf{60} (2021), no. 6, Paper No. 213.
		
		
		\bibitem{Hornung-Moser-2012-APDE}
		\textsc{P. Hornung and R. Moser}, \emph{Energy identity for intrinsically biharmonic maps in four dimensions}. Anal. PDE \textbf{5} (2012), no. 1, 61-80.
		
		\bibitem{Lamm-Riviere-2008}
		\textsc{T. Lamm and T. Rivi\`ere,} \emph{Conservation
			laws for fourth order systems in four dimensions.} Comm. Partial Differential
		Equations \textbf{33} (2008), 245-262.
		
		\bibitem{l-r-4}
		\textsc{P. Laurain and T. Rivi\`ere,} \emph{Energy quantization for biharmonic maps.} Adv. Calc. Var. \textbf{6} (2013), no. 2, 191--216.
		
		\bibitem{l-r-2}
		\textsc{P. Laurain and T. Rivi\`ere,} \emph{Angular energy quantization for linear elliptic systems with antisymmetric potentials and applications.} Anal. PDE \textbf{7} (2014), 1-41.
		
		\bibitem{Lin-1999-Annals} \textsc{F. H. Lin}, \emph{Gradient estimates
			and blow-up analysis for stationary harmonic maps.} Ann. of Math.
		(2) \textbf{149} (1999), no. 3, 785-829.
		
		\bibitem{Lin-Riviere-2002}\textsc{F. H. Lin and T. Rivi\`ere}, \emph{Energy
			quantization for harmonic maps}. Duke Math. J. \textbf{111} (2002),
		no. 1, 177-193.
		
		\bibitem{LinWang1998}\textsc{F. H. Lin, C. Y. Wang}, \emph{Energy identity of harmonic map flows from surfaces at finite singular time}.
		Calc. Var. Partial Differential Equations \textbf{6} (1998), no. 4, 369-380.
		
		\bibitem{Liu-Yin-2016-MZ}
		\textsc{L. Liu and H. Yin}, \emph{Neck analysis for biharmonic maps}. Math. Z. \textbf{283} (2016), no. 3-4, 807-834.
		
\bibitem{Mattila-book} \textsc{P. Mattila}, \emph{Geometry of sets
and measures in Euclidean spaces.} \emph{Fractals and rectifiability.} 		Cambridge Stud. Adv. Math., 44 Cambridge University Press, Cambridge, 1995.

		
		\bibitem{Naber-Val-2017-AnnMath} \textsc{A. Naber and D. Valtorta,}
		\emph{Rectifiable-Reifenberg and the regularity of stationary and
			minimizing harmonic maps.} Ann. of Math. (2) \textbf{185 }(2017),
		131-227.
		
		
		\bibitem{Naber-V-24-arXiv}\textsc{A. Naber and D. Valtorta}, \emph{Energy Identity for Stationary Harmonic Maps}. Preprint at arXiv:2401.02242[math.AP].
		
		\bibitem{Parker1996}  \textsc{T. Parker}, \emph{Bubble tree convergence for harmonic maps}
		J. Differential Geom. \textbf{44} (1996), no. 3, 595-633.
		
		\bibitem{QingTian1997} \textsc{J. Qing, G. Tian}, \emph{Bubbling of the heat flows for harmonic maps from surfaces}.
		Comm. Pure Appl. Math. \textbf{50} (1997), no. 4, 295-310.
		
		
		\bibitem{Riviere-2007} \textsc{T. Rivi\`ere,} \emph{Conservation
			laws for conformally invariant variational problems.} Invent. Math.
		\textbf{168} (2007), 1-22.
		
		\bibitem{Sacks-Uhlenbeck-1981}
		\textsc{J. Sacks and K. Uhlenbeck}, \emph{The existence of minimal immersions of 2-spheres}. Ann. of Math. (2) \textbf{113} (1981), no. 1, 1-24.
		
		\bibitem{Scheven-2008-ACV}\textsc{C. Scheven}, \emph{Dimension reduction for the singular set of biharmonic maps}. Adv. Calc. Var. \textbf{1} (2008), no. 1, 53-91.
		
		\bibitem{Simon-book} \textsc{L. Simon,} \emph{Theorems on the regularity
			and singularity of minimal surfaces and harmonic maps.} Lectures on
		geometric variational problems (Sendai, 1993), 115-150, Springer,
		Tokyo, 1996.
		
		
		\bibitem{Struwe-2008} \textsc{M. Struwe,} \emph{Partial regularity
			for biharmonic maps, revisited.} Calc. Var. Partial Differential Equations
		\textbf{33} (2008), 249-262.
		
		\bibitem{Strzelecki-2003}
		\textsc{P. Strzelecki}, \emph{On biharmonic maps and their generalizations}. Calc. Var. Partial Differential Equations \textbf{18} (2003), no. 4, 401-432.
		
		\bibitem{Wang-2004-CV} \textsc{C. Y. Wang,} \emph{Remarks on biharmonic
			maps into spheres.} Calc. Var. Partial Differential Equations \textbf{21}
		(2004), 221-242.
		
		\bibitem{Wang-2004-MZ} \textsc{C. Y. Wang,} \emph{Biharmonic maps
			from $\mathbb{R}^4$ into a Riemannian manifold.} Math. Z. \textbf{247} (2004),
		65-87.
		
		\bibitem{Wang-2004-CPAM} \textsc{C. Y. Wang,} \emph{Stationary biharmonic
			maps from $R^m$ into a Riemannian manifold.} Comm. Pure Appl. Math. \textbf{57}
		(2004), 419-444.
		
		\bibitem{Wang-Zheng-2012}
		\textsc{C. Y. Wang and S. Z. Zheng,} \emph{Energy identity of approximate biharmonic maps to Riemannian manifolds and its application.} J. Funct. Anal. \textbf{263} (2012), no. 4, 960-987.
		
		\bibitem{Wang-Zheng-2013}
		\textsc{C. Y. Wang and S. Z. Zheng,} \emph{Energy identity for a class of approximate biharmonic maps into sphere in dimension four}. Discrete Contin. Dyn. Syst. \textbf{33} (2013), no. 2, 861-878.
		
		\bibitem{zhang-zhu-2026}
		\textsc{H. Zhang and M. Zhu}, \emph{Energy quantization for stationary harmonic maps into homogeneous spaces}. Sci. China Math. \textbf{69} (2026), no. 1, 167-182.
	\end{thebibliography}
\end{document}